\documentclass[10pt]{amsart}

\usepackage{amsmath}
\usepackage{amssymb}
\usepackage{mathrsfs}
\usepackage{ stmaryrd }
\usepackage[all,cmtip]{xy}
\usepackage{pigpen}
\usepackage{graphicx}
\usepackage{enumitem}
\usepackage{hyperref}
\usepackage{verbatim}

\usepackage{mathtools}
\makeatletter

\newcommand{\xRrightarrow}[2][]{\ext@arrow 0359\Rrightarrowfill@{#1}{#2}}
\newcommand{\Rrightarrowfill@}{\arrowfill@\equiv\equiv\Rrightarrow}
\newcommand{\xLleftarrow}[2][]{\ext@arrow 3095\Lleftarrowfill@{#1}{#2}}
\newcommand{\Lleftarrowfill@}{\arrowfill@\Lleftarrow\equiv\equiv}
\makeatother
\numberwithin{equation}{subsection}

\newtheorem{prop}[subsection]{Proposition}
\newtheorem{thm}[subsection]{Theorem}

\newtheorem*{thm*}{Theorem}

\newtheorem*{lem*}{Lemma}
\newtheorem{cor}[subsection]{Corollary}

\theoremstyle{definition}

\theoremstyle{remark}

\newtheorem*{assump*}{Assumption}

\theoremstyle{definition}
\newtheorem{definition}[subsubsection]{Definition}

\theoremstyle{remark}
\newtheorem{remark}[subsubsection]{Remark}

\newcommand{\C}{\ensuremath{\mathbb{C}}}

\newcommand{\ad}{\ensuremath{\mathrm{ad}}}
\newcommand{\supp}{\ensuremath{\mathrm{supp}}}
\newcommand{\lps}{[\![}
\newcommand{\rps}{]\!]}
\newcommand{\llps}{(\!(}
\newcommand{\lrps}{)\!)}

\renewcommand{\O}{\ensuremath{\mathcal{O}}}

\newcommand{\G}{\ensuremath{\mathcal{G}}}

\newcommand{\Res}{\ensuremath{\mathrm{Res}}}

\newcommand{\SL}{\ensuremath{\mathrm{SL}}}
\newcommand{\SU}{\ensuremath{\mathrm{SU}}}

\newcommand{\Gr}{\ensuremath{\mathrm{Gr}}}
\newcommand{\Fl}{\ensuremath{\mathrm{Fl}}}

\newcommand{\fka}{\ensuremath{\mathfrak{a}}}

\newcommand{\fkf}{\ensuremath{\mathfrak{f}}}

\newcommand{\fkh}{\ensuremath{\mathfrak{h}}}

\newcommand{\fks}{\ensuremath{\mathfrak{s}}}
\newcommand{\fkt}{\ensuremath{\mathfrak{t}}}

\newcommand{\bbA}{\ensuremath{\mathbb{A}}}

\newcommand{\bbC}{\ensuremath{\mathbb{C}}}

\newcommand{\bbG}{\ensuremath{\mathbb{G}}}

\newcommand{\bbQ}{\ensuremath{\mathbb{Q}}}
\newcommand{\bbR}{\ensuremath{\mathbb{R}}}

\newcommand{\bbZ}{\ensuremath{\mathbb{Z}}}

\newcommand{\dom}{\ensuremath{\mathrm{dom}}}

\newcommand{\calA}{\ensuremath{\mathcal{A}}}
\newcommand{\calB}{\ensuremath{\mathcal{B}}}

\newcommand{\calG}{\ensuremath{\mathcal{G}}}
\newcommand{\calH}{\ensuremath{\mathcal{H}}}

\newcommand{\calO}{\ensuremath{\mathcal{O}}}

\newcommand{\calT}{\ensuremath{\mathcal{T}}}
\newcommand{\calU}{\ensuremath{\mathcal{U}}}

\newcommand{\calX}{\ensuremath{\mathcal{X}}}

\newcommand{\af}{\ensuremath{\mathrm{af}}}

\newcommand{\lleq}{\ensuremath{\preccurlyeq}}
\newcommand{\ggeq}{\ensuremath{\succcurlyeq}}

\newcommand{\po}{\ar@{}[dr]|{\text{\pigpenfont R}}}
\newcommand{\pb}{\ar@{}[dr]|{\text{\pigpenfont J}}}

\usepackage{color}

\begin{document}
\title{On the smooth locus of affine   Schubert varieties }
\author{Georgios Pappas}
\address{Dept. of Mathematics, Michigan State University, E. Lansing, MI 48824, USA}
\email{pappasg@msu.edu}

\author{Rong Zhou}
\address{Department of Pure Mathematics and Mathematical Statistics, University of Cambridge, Cambridge, UK, CB3 0WA}
\email{rz240@dpmms.cam.ac.uk}

\begin{abstract} 
We give a simple and uniform proof of a conjecture of Haines--Richarz characterizing the smooth locus of Schubert varieties in twisted affine Grassmannians. Our method is elementary and avoids any representation theoretic techniques, instead relying on a combinatorial analysis of tangent spaces of Schubert varieties.
\end{abstract}

\maketitle
%%
%% The title of the paper goes here.  Edit to your title.
%%

%%
%% Now edit the following to give your name and address:
%% 

\date{\today}

	%\makeatletter
%\@namedef{subjclassname@2020}{\textup{2020} Mathematics Subject Classification}
%\makeatother

%\keywords{}
%\subjclass[2020]{}
\tableofcontents

\section{Introduction}
Schubert subvarieties of flag varieties, classical or affine, and their singularities have been a subject of considerable interest for  some time. A basic problem is to characterize  the Schubert varieties which are smooth or, more generally, identify their smooth locus; see, for example, \cite{BiLaksh} for many results in the classical case. In the affine case, this problem also relates,
via the theory of local models \cite{PRS}, to the study of singularities of reductions of Shimura varieties modulo primes. For example, understanding when affine Schubert varieties are smooth leads to a classification of Shimura varieties with parahoric level at a prime, whose reduction at that prime is smooth, see \cite{HPR}.

It turns out that this problem is simpler for Schubert varieties in  affine Grassmannians.
Indeed, it was first shown by Evens--Mirkovic \cite{EM} that, for Schubert varieties in affine Grassmannians of split groups, the smooth locus is simply the `obvious' open orbit. This result follows from their calculation of the characteristic cycle of the intersection cohomology sheaf of the Schubert variety and easily implies a characterization of the Schubert varieties which are smooth. Affine Grassmannians are also objects of Kac-Moody theory: The split case above is obtained from `untwisted' affine Kac--Moody Lie algebras   but one can equally consider  the general (twisted) affine case, cf. \cite[Ch.\,13]{KumarBook}. The corresponding (twisted) affine Grassmannians can also be constructed using  loop groups obtained from reductive groups over the field $\bbC\llps t \lrps$ of Laurent power series, see \cite{PR}, also for the comparison between the two points of view. In this group-theoretic set-up, Haines and Richarz gave a classification of all Schubert varieties in twisted affine Grassmannians which are smooth and also conjectured a generalization of the Evens--Mirkovic result \cite[Conjecture 5.4]{HainesRicharz}. The result of Haines and Richarz was used in the classification of Shimura varieties with smooth parahoric reduction of \cite{HPR}.

In this paper, we give a simple combinatorial proof of both the original result of Evens--Mirkovic and of its extension conjectured by Haines--Richarz. Let us  now give some details. We formulate our results using the theory of reductive groups over the discretely valued field $\bbC\llps t \lrps$. Let $G$ be a connected reductive group over  $F=\bbC\llps t \lrps$. Associated to any special vertex $\fks$ in the Bruhat--Tits building for $G$, we have an associated group scheme $\calG$ over $\bbC\lps t\rps$. We let $\Gr_{\calG}:=LG/L^+\calG$ denote the twisted affine Grassmannian associated to $\calG$, see \cite{PR}.  
We fix $T\subset B\subset G$ a maximal torus and Borel pair in $G$ defined over $F$ which are in good position relative to $\fks$, and let $I=\mathrm{Gal}(\overline{F}/F)$. Then to each dominant element  $\mu\in X_*(T)_I^+$, we have an associated point $t^\mu\in \Gr_{\calG}(\bbC)$ and the  associated (affine) Schubert variety $\Gr_{\calG,\lleq\mu}$. By definition $\Gr_{\calG,\lleq\mu}$ is the closure of the $L^+\calG$-orbit  $\Gr_{\calG,\mu}$ of $t^\mu$. Our main theorem is the following, which verifies \cite[Conjecture 5.4]{HainesRicharz}.

\begin{thm}\label{introthm: main}
	Let $\fks$ be an absolutely special vertex. Then the smooth locus of $\Gr_{\calG,\lleq\mu}$ is $\Gr_{\calG,\mu}$.
\end{thm}
As we mentioned above, the result is well-known in the case of split groups due to work of Evens--Mirkovic \cite{EM},  see also Malkin--Ostrik--Vybornov \cite[Cor. B]{MOV}. When the group $G$ is not split, the picture is complicated  by a phenomenon which was first observed in the theory of local models for integral models of Shimura varieties by Richarz (\cite[Proposition 4.6]{Arz}): If the vertex $\fks$ is special but not absolutely special, it can happen that the smooth locus of a Schubert variety is larger than the open stratum. This implies the existence of certain Shimura varieties with `exotic good reduction' (i.e. Shimura varieties with non-hyperspecial level at a prime number $p$ which, nevertheless, have good reduction at $p$, see \cite[\S 5]{HPR}). Our result verifies  the expectation that this does not happen in any other case.  Implicitly, it also simplifies the proofs of the classification result of \cite{HPR}.

This conjecture has also been considered in work of Besson--Hong \cite{BH}, who verified it when $G_{\ad}$ does not have factors of type $C-BC_n$. Their method is very indirect and relies on difficult representation theoretic techniques. In particular, it crucially uses an extension of a result of Zhu \cite{ZhuTfixed} which gives a geometric interpretation of the Frenkel--Kac--Segal isomorphism between basic representations of affine Lie algebras and lattice vertex algebras. Furthermore, in the case of $E_6$, the proof requires computer assistance. Combined with  Richarz's thesis, which analyses the $C-BC_n$ case explicitly, this can be used to deduce the conjecture in general. In contrast, our method is  simpler, avoiding any  use of representation theory and computer aided computations, and can be applied uniformly to all cases. Specialized to the case of split groups, it gives a new proof of the result in  \cite{EM} and \cite{MOV}.  Moreover, Theorem \ref{introthm: main}, combined with now standard techniques, can be used to classify the smooth locus of twisted affine Grassmannians when $\bbC$ is replaced by an algebraically closed field of any characteristic, at least when the group is tamely ramified (cf. \cite[Corollary 6.1]{HainesRicharz}.)%, \cite{Lou}).} 

We now give some details of the proof. Since the smooth locus is open, it suffices to show that when $\fks$ is absolutely special, $\Gr_{\calG,\lleq\mu}$ is singular along the irreducible components of the boundary. These are given by  $\Gr_{\calG,\lleq\lambda}$, where $\mu\rightsquigarrow\lambda$ is a minimal degeneration (cf. \S\ref{subsec: min deg}).
To prove this, we consider a lower bound for the dimension of the tangent space to $\Gr_{\calG,\lleq\mu}$ at $t^\lambda$.  We obtain  such a bound by considering the tangent vectors arising from smooth curves in the Schubert variety coming from affine root groups; we call these  \textit{root curves}. We show that for a minimal degeneration, the lower bound we obtain is equal to the dimension of the Schubert variety; here we use Stembridge's classification of minimal degenerations and some simple computations using the \'echelonnage root system. When $\fks$ is absolutely special,  conjugating one of the root curves by an element of the opposite unipotent gives an extra tangent vector which makes the point non-smooth; this extra tangent direction has a component `along the Cartan direction' (cf. Proposition \ref{prop: tangent space AFV}). The assumption that $\fks$ is absolutely special is crucial here; in the case of exotic good reduction the shape of the neighbouring affine root groups for the special but not absolutely vertex prevents us from producing extra tangent directions. From this point of view, one can think of the singularities of minimal degenerations as being caused by extra tangent directions along the Cartan. 
Roughly speaking, the moral of the main Theorem of \cite{ZhuTfixed} is similar: Properties of $\Gr_{\calG}$ can be studied by restriction to the   $0$-dimensional ind-subscheme $\Gr_{\calT}\subset \Gr_{\calG}$, where $\calT\subset\calG$ is a maximal torus.

We now give an outline of the paper. In section 2, we give an explicit  description of the tangent space of partial affine flag varieties at the base point in terms of affine roots. This should be well-known to experts, but we provide the details here as they are needed for later. In section 3, we use the tangent vectors from root curves to give a lower bound on the dimension of tangent spaces of Schubert varieties, and in section 4 we apply the bound to the boundary components of minimal degenerations to prove the main result.  At the end of the section, we give an alternative, less combinatorial, argument for the proof. This does not even need the full classification of minimal degenerations due to Stembridge but uses some additional algebraic geometric input. We finish with an example of the method applied to the case of the quasi-minuscule Schubert variety in the ramified triality, which gives a short proof of a  result of Haines--Richarz \cite[\S8]{HainesRicharz}.

\textit{Acknowledgements:} G.\,P. is supported by NSF grant \#DMS-2100743. R.\,Z. is supported by  EPSRC grant ref. EP/Y030648/1 as part of the ERC Starting Grant guarantee scheme. The authors would like to thank  Thomas Haines, Jo\~ao Louren\c{c}o, Michael Rapoport and Timo Richarz for comments on an earlier version of this paper.

\section{Tangent spaces of partial affine flag varieties}

We give a description of the tangent space of partial affine flag varieties at the base point in terms of the associated affine root system. Such a description  can also be  deduced from the relationship between the affine flag varieties and Kac--Moody algebras as in \cite[Proposition 4.5]{HLR}. However our construction of the isomorphism in Proposition \ref{prop: tangent space AFV} and its relation with affine root groups will be useful for our purposes.

\subsection{}Let $F=\bbC(\!(t)\!)$ and $G/F$ be an adjoint, absolutely  simple reductive group.
Let $H$ be the split Chevalley form of  $G$ and let $F':=\bbC(\!(u)\!)/F$ be the minimal splitting field for $G$ of degree $e$ with $u^e=t$. Let $\tau$ be a generator of $I:=\mathrm{Gal}(F'/F)$; then $\tau(u)=\zeta u$ where $\zeta$ is an $e^{\mathrm{th}}$ root of unity. Since $G$ is absolutely simple, we have $e=1,2$ or $3$. For any  affine  group scheme $X$ over $F$, we write $LX$ for the loop group scheme over $\C$ associated to $X$. Similarly for $\calX$  an affine  group scheme over $\O_F:=\bbC\lps t\rps$, we write $L^+\calX$ for the positive loop group scheme over $\C$ associated  to $\calX$, see \cite{PR}.

Let $H$ denote the split Chevalley form over $\bbC$ of $G$. Upon fixing a Chevalley pinning $(B_H,T_H,\{X_{\alpha'}\})$ of $H$, we may realize $G$ as the functor whose value on an $F$-algebra $R$ is given by 
$$
G(R):=\{g\in H(R\otimes_FF')\, |\, \sigma(g)=g\}
$$
 where $\sigma=\sigma_0\circ \tau$, and  $\sigma_0$ is a diagram automorphism of the Dynkin diagram of $H$, which we consider as an element of $\mathrm{Aut}(H)$ via the pinning. 

The torus $T_H$ determines a maximal torus $T$ of $G$, and we let $S$ denote the maximal $F$-split subtorus of $T$, which is a maximal $F$-split torus of $G$. We let $B\supset T$ be the Borel subgroup defined over $F$  determined by $B_H$ and let $X_*(T)_I^+\subset X_*(T)_I$ denote the set of dominant elements with respect to the choice of $B$.

\subsection{}\label{para: affine root gps}We recall the description of the affine root system for $G$ and $H$ following \cite{PR1}. Let $\Phi$ denote the set of relative roots for $G$ and $\Phi_{\af}$ the set of affine roots.  Similarly,  we let $\Phi_{H}$ denote the set of absolute roots for $H$, and $\Phi_{H,\af}$ the set of affine roots. Let $\calA(G,S,F)$ (resp. $\calA(H,T_H,F')$) denote the apartment for $S$ over $F$ (resp. $T_H$ over $F'$). Then the affine roots $\Phi_{\af}, \Phi_{H,\af}$ are affine functions on $\calA(G,S,F)$, $\calA(H,T_H,F')$ respectively.

	\begin{definition} A vertex $\fks\in \calA(G,S,F)$ is said to be \textit{absolutely special} if its image under the simplicial embedding $\calA(G,S,F)\rightarrow \calA(H,T_H,F')$ is a special (equivalently hyperspecial) vertex.
\end{definition}
\begin{remark}This is a correction of \cite[Definition 5.1]{HainesRicharz}. It is necessary to take $F'$ the \textit{minimal} splitting field of $G$ since for $G={\rm PU}_{2n+1}$, the vertex labelled `s' in \cite[equation (5.1)]{HainesRicharz} becomes special over a degree 4 ramified extension of $F$. We refer to Remark \ref{rem: general group} for the definition of absolutely special vertex in general.
\end{remark}

The action of $\sigma$ on $H_{F'}$ induces an action on $\calA(H,T_H,F')$ which fixes the hyperspecial vertex corresponding to $H_{\calO_{F'}}$, and hence descends to a point $\fks_0\in \calA(G,S,F)$. Then $\fks_0$ is an absolutely special vertex in the sense above. 
The choice of $\fks_0$ induces identifications 
$$
\calA(G,S,F)\cong X_*(T)_I\otimes_{\bbZ}\bbR\cong X_*(S)\otimes_{\bbZ}\bbR,
$$
$$
\calA(H,T_H,F')\cong X_*(T_H)\otimes_{\bbZ}\bbR.
$$
Using these identifications, we can write affine roots $a\in \Phi_{\af}, a'\in \Phi_{H,\af}$  as $a=\alpha+m$, $a'=\alpha'+m'$, where $\alpha\in \Phi$ and $\alpha'\in \Phi_H$ and $m,m'\in \bbQ$.  We normalize the valuation $v$ on $F'$ so that $v(u)=\frac{1}{e}$. Then we have 
$$
\Phi_{H,\af}=\{\alpha'+m'\,|\,\alpha'\in \Phi_H, m'\in\frac{1}{e}\bbZ\}.
$$
For $a'=\alpha'+m'\in \Phi_{H,\af}$, the corresponding root subgroup  $\calU'_{a'}\subset LH$  over $\bbC$ corresponds to the subspace $\bbC X_{\alpha'}u^m$ of the Lie algebra of the loop group $LH$, where $X_{\alpha'}\in \mathrm{Lie}H$ is our fixed Chevalley generator.

Let $\alpha\in \Phi$  be a relative root,  then $\alpha$ determines a $\sigma$-orbit $C_\alpha\subset \Phi_H$ of absolute roots. We write $d_\alpha=|C_\alpha|$.  We define $$R_\alpha=\begin{cases} \frac{1}{d_\alpha}\bbZ & \text{if $\frac{1}{2}\alpha \notin \Phi$,}\\
\frac{1}{2d_\alpha}+\frac{1}{d_\alpha}\bbZ & \text{if $\frac{1}{2}\alpha\in \Phi$}.
\end{cases}$$
Then by \cite[Proposition 2.2]{PR1}, we have
$$
\Phi_{\af}=\{\alpha+m\,|\,\alpha\in \Phi, m\in R_\alpha\}.
$$

\subsection{} 

For $\alpha\in \Phi$ we write $U_\alpha\subset G$  for the corresponding relative root subgroup over $F$, and
for $a=\alpha+m\in \Phi_{\af}$, we let $\calU_{a}\subset LU_\alpha$ denote the corresponding root subgroup over $\bbC$ defined in \cite{PR}. For each $a\in \Phi_{\af}$ we recall the description of $\calU_a$ below. We then use this to define a morphism $\tilde{f}_a:\bbA^1\rightarrow \calU_a$ and a corresponding tangent vector $e_a$ in the Lie algebra of $LH$ fixed by $\sigma$.

Case (1): $\frac{1}{2}{\alpha},2\alpha\notin \Phi$. Then all elements of $C_\alpha$ are orthogonal to one another, and our choice of $X_{\alpha'}$ determines an isomorphism $x_\alpha:U_{\alpha}\cong \mathrm{Res}_{K/F}\bbG_a$ for $K/F$ of degree $d=d_\alpha=[K:F]$, where $d|e$. Under our assumptions, we have $d=1$ or $d=e$, and hence $K=F$ or $K=F'$.
Let $a=\alpha+m\in \Phi_{\af}$. Then $m\in  \frac{1}{d}\bbZ$, and $$x_\alpha(\calU_a(\bbC))=\bbC u^{em}.$$ We define an isomorphism $\tilde{f}_a:\bbA^1\xrightarrow{\sim}\calU_a$ such that its composition with $x_\alpha$ is given by $r\mapsto ru^{em}.$ The corresponding subspace of the Lie algebra of $LH$ is generated by

$$
e_a=\sum_{i=1}^{d} X_{\sigma_0^i(\alpha')}(\zeta^iu)^{em},
$$
 where $\alpha'\in C_\alpha$.

Case (2): $\alpha,2\alpha\in \Phi$.  In this case $F'/F$ is a degree 2 extension, and we have an isomorphism $x_\alpha:U_\alpha\cong M_{F',F}$ where $M_{F',F}$ is the group scheme with $F$-points given by
$$M_{F',F}(F)=\{(c,d)\,|\,c\sigma(c)+d+\sigma(d)=0\}.$$
Moreover, $U_{2\alpha}$ is the subgroup of $U_{\alpha}$ which corresponds under $x_\alpha$ to 
$$
\{(0,d)\,|\,d+\sigma(d)=0\}\subset M_{F', F}.
$$

(2a) Let $a=\alpha+m\in \Phi_{\af}$. Then $m\in \frac{1}{2}\bbZ$, and we have $$x_\alpha(\calU_{\alpha}(\bbC))=\{(ru^{2m},\frac{(-1)^{2m+1}}{2}r^2 u^{4m})| r\in \bbC\}.$$We let $\tilde{f}_a:\bbA^1\xrightarrow{\sim} \calU_a$ be  such that 
$$
x_\alpha\circ \tilde{f}_a(r)=(ru^{2m},\frac{(-1)^{2m+1}}{2}r^2 u^{4m}).
$$
 The corresponding subspace of the Lie algebra of $LH$ is generated by $$e_a=u^{2m}(X_{\alpha'}+(-1)^{2m}X_{\sigma_0(\alpha')}),$$ where $\alpha'\in C_\alpha$.

(2b) Let $a=2\alpha+m$. Then $m\in \frac{1}{2}+\bbZ$, and we have 
$$
x_{2\alpha}(\calU_{a}(\bbC))=\{(0,ru^{2m})|r\in \bbC\}.
$$
We let $\tilde{f}_a:\bbA^1\xrightarrow{\sim} \calU_a$ be  such that 
$$
x_\alpha\circ \tilde{f}_a(r)=(0,ru^{2m}).
$$
The corresponding element of the Lie algebra of $LH$ is $$e_a=u^{2m}X_{\alpha'}$$ for $\alpha'\in C_{2\alpha}$.

\subsection{}
Let $\fkh$ denote the Lie algebra of $H$ which decomposes as 
$$
\fkh\cong (\bigoplus_{\alpha'\in \Phi_H} \bbC X_{\alpha'})\oplus\fkt_H,
$$
 where $\fkt_H$ is the Lie algebra of $T_H$. This gives a corresponding decomposition of the vector space $$\fkh\otimes_{\bbC}\bbC[u,u^{-1}]=\left(\bigoplus_{\alpha'\in \Phi_H}\bbC[u,u^{-1}]X_{\alpha'}\right)\oplus (\fkt_H\otimes_{\bbC}\bbC[u,u^{-1}])\subset \mathrm{Lie}LH.$$
 This space is equipped with an action of $\sigma$, which preserves  $\bigoplus_{\alpha'\in \Phi_H}\bbC[u,u^{-1}]X_{\alpha'}$ and $\fkt_H\otimes_{\bbC}\bbC[u,u^{-1}]$.
\begin{prop}\label{prop: basis of invariants}
The vector space  $(\bigoplus_{\alpha'\in\Phi_H}\bbC[u,u^{-1}]X_{\alpha'})^\sigma$ has a basis given by the set $\calB=\{e_a\}_{a\in \Phi_{\af}}$.
\end{prop}
\begin{proof} Clearly, the $e_a$ are $\sigma$-invariant and linearly independent. 
	
Note that the vector space $ \bbC[u,u^{-1}]X_{\alpha'}$  has a basis given by 
$$
\{X_{\alpha'}u^m|\alpha'\in \Phi_H,m\in \bbZ\}.
$$ 

Suppose $v\in(\bigoplus_{\alpha'\in\Phi_H}X_{\alpha'}\otimes \bbC[u,u^{-1}])^\sigma$ has a non-zero component $X_{\alpha'}u^m$; upon scaling we may assume this coefficient is 1. Let $\alpha\in \Phi_{\af}$ denote the associated relative root. We show $v$ is spanned by the $e_a$ by induction on the number of non-zero coefficients.

Case (1): $\frac{1}{2}{\alpha},2\alpha\notin \Phi$. Let $d=|C_\alpha|$ as above. Then by $\sigma^d$-invariance, we have  $\frac{e}{d}|m$, and the coefficient of $X_{\sigma_0^i(\alpha')}u^m$ in $v$ is $\zeta^{im}$ for $i=1,\dotsc,d$. Set $a=\alpha+\frac{1}{e}m\in \Phi_{\af}$. Then $v-e_a$ is spanned by $\calB$ by induction.

Case (2a): $2\alpha\in \Phi$. In this case $C_\alpha=\{\alpha',\sigma_0(\alpha')\}$, and the coefficient of $X_{\sigma_0(\alpha')}u^m$ is $(-1)^{m}u^m$. Then for $a=\alpha+\frac{1}{2}m$, $v-e_a$ is spanned by $\calB$ by induction.

Case (2b): $\frac{1}{2}\alpha\in \Phi$. In this case $C_\alpha=\{\alpha'\}$, and the $\sigma$-invariance implies $m$ is odd. Set $a=\alpha+m$; then $v-e_a$ is spanned by $\calB$ by induction. \end{proof}

\subsection{} The following computation will be used in \S\ref{sec: min deg} to produce extra tangent directions along the Cartan.

\begin{prop}\label{prop: Cartan direction}

\begin{enumerate}
	\item Let $a=\alpha+m$ with $\frac{1}{2}\alpha,2\alpha\notin\Phi$. Set $b=-\alpha-m-\frac{1}{d_\alpha}\in \Phi_{\af}$ and $h=\tilde{f}_b(1)$. Then $\mathrm{Ad}(h) e_a$ has a non-zero component in $(\fkt_H\otimes u^{-1}\bbC[u^{-1}])^\sigma$.
	
	\item Let $a=2\alpha+m$ with $\alpha\in \Phi$ so that $2m$ is an odd integer. Set $b=-\alpha-\frac{m}{2}-\frac{1}{2}\in \Phi_{\af}$ and $h=\tilde{f}_b(1)$. Then $\mathrm{Ad}(h)e_a$ has a non-zero component in $(\fkt_H\otimes u^{-1}\bbC[u^{-1}])^\sigma$.
\end{enumerate}

\end{prop}

\begin{proof}In each case let $G_\alpha$ denote the simply-connected cover of the subgroup of $G$ generated by $U_\alpha$ and $U_{-\alpha}$. Then $G_\alpha\cong \Res_{K/F}\SL_2$ for some $K/F$ finite	 in (1) and $G_\alpha\cong \SU_3$ the special unitary group for a degree 2 ramified extension $F'/F$ in (2). Since the affine root groups contained in $U_\alpha,U_{-\alpha}$ arise from $G_\alpha$, and the natural map $LG_\alpha\rightarrow LG$ induces an injection on tangent spaces preserving Cartan spaces, it suffices to prove (1) and (2) for the group $G_\alpha$. In these cases, we  compute using the standard matrix representations of these groups. Moreover, upon possibly changing the choice of absolutely special vertex $\fks_0$, we may assume $m=0$ in (1) and $m=-\frac{1}{2}$ in (2).
	
	(1) Suppose $K=F(u')$, so that $u^{\frac{e}{d_\alpha}}=u'$. We have 
	$$
	e_a=\left(\begin{matrix} 0 & 1 \\ 0 & 0  \end{matrix}\right)\text{ and } h=\left(\begin{matrix} 1 & 0 \\ u'^{-1}&1 \end{matrix}\right).
	$$
	Then we have 
	\begin{align*}\mathrm{Ad}(h)e_a&= \left(\begin{matrix} 1 & 0 \\u'^{-1}& 1
	\end{matrix}\right) \left(\begin{matrix} 0& 1\\ 0& 0
	\end{matrix}\right)\left(\begin{matrix} 1 & 0 \\-u'^{-1}& 1
	\end{matrix}\right)
	\\ &=\left(\begin{matrix} -u'^{-1} & 1 \\  -u'^{-2} & u'^{-1}
	\end{matrix}\right). \end{align*}
The diagonal component of $\mathrm{Ad}(h)e_a$ gives a non-zero component in $(\fkt_H\otimes u^{-1}\bbC[u^{-1}])^\sigma$.
	
	(2)  	In this case, $e=d_\alpha=2$.  Then we have $$e_a=\left(\begin{matrix}
0 & 0 & u^{-1}\\ 0 & 0& 0 \\ 0 & 0 & 0
	\end{matrix}\right)\text{ and } h= \left(\begin{matrix}
	1 & 0   & 0 \\ 1 & 1& 0 \\ -\frac{1}{2}& -1& 1
	\end{matrix}\right).$$
	Then we compute \begin{align*}\mathrm{Ad}(h) e_a&= \left(\begin{matrix}
	1 & 0   & 0 \\ 1& 1& 0 \\  -\frac{1}{2} & -1& 1
	\end{matrix}\right) \left(\begin{matrix}
	0 & 0 & u^{-1}\\ 0 & 0& 0 \\ 0 & 0 & 0
	\end{matrix}\right)\left(\begin{matrix}
	1 & 0   & 0 \\ -1 & 1& 0 \\  -\frac{1}{2}& 1& 1
	\end{matrix}\right)\\
	&=\left(\begin{matrix}-\frac{u^{-1}}{2} & u^{-1}   & u^{-1} \\ -\frac{u^{-1}}{2} & u^{-1}& u^{-1} \\  -\frac{u^{-1}}{4}& -\frac{u^{-1}}{2}& -\frac{u^{-1}}{2}
	\end{matrix}\right).
	\end{align*}
	The diagonal of $\mathrm{Ad}(h)e_a$ gives a non-zero component in $(\fkt_H\otimes u^{-1}\bbC[u^{-1}])^\sigma$.
	\end{proof}

\subsection{} Now let $\fkf\subset \calA(G,S,F)$ be a facet. Then $\fkf$ determines a facet in $\calA(H,T_H,F')$ also denoted $\fkf$. We let $\calG_{\fkf}$ and $\calH_{\fkf}$ denote the parahoric group schemes  over $\calO_F$ and $\calO_{F'}$ respectively corresponding to $\fkf$. We let $\Fl_{G,\fkf}:=LG/L^+\calG_{\fkf}$ denote the partial affine flag variety associated to $\calG_{\fkf}$, and $e\in \Fl_{G,\fkf}(\bbC)$ the base-point. We have the following description of the tangent space $T_e\Fl_{G,\fkf}$  to $\Fl_{G,\fkf}$ at the base-point.

\begin{prop}\label{prop: tangent space AFV}
	There is an isomorphism
	$$
	T_e\Fl_{G,\fkf}\cong \left(\bigoplus_{a\in\Phi_{\af}, a(\fkf)<0}\bbC e_a\right)\oplus(\fkt_H\otimes u^{-1}\bbC[u^{-1}])^\sigma.
	$$	Here, we write $a(\fkf)<0$ if $a(x)<0$ for some (equivalently) all $x\in \fkf$.
\end{prop}
\begin{proof}
	By \cite[Corollary 3.9]{HLR}, we have an isomorphism $T_e{\Fl_{G,\fkf}}\cong (T_eL^{--}\calH_{\fkf})^\sigma$, where $L^{--}\calH_{\fkf}$ denotes the negative loop group as defined in \cite[\S3.3]{HLR}. Let $\fka'\subset \calA(H,T_H,F')$ be an alcove whose closure contains $\fkf$ with associated Iwahori group scheme $\calH_{\fka'}$. Then by definition of $L^{--}\calH_{\fkf}$ there is an isomorphism $$T_eL^{--}\calH_{\fkf}\cong \bigcap_{w\in W_{H,\fkf}}{^w(T_eL^{--}\calH_{\fka'}}).$$
	Here $W_{H,\fkf}$ denotes the subgroup of the Iwahori Weyl group $W_H$ for $H_{F'}$ generated by the simple reflections fixing $\fkf$. Recall we identify $\Phi_{H,\af}$ with the set of pairs $(a',m), a'\in \Phi_H, m\in \bbZ$ using the choice of absolutely special vertex $\fks_0$.
	We have an isomorphism 
	$$
	T_eL^{--}\calH_{\fka'}\cong  \left(\bigoplus_{a'=\alpha'+m\in \Phi_{H,\af},a'(\fka')<0}\bbC X_{\alpha'+m}  \right)\oplus (\fkt_H\otimes u^{-1}\bbC[u^{-1}]).
	$$
	Since 
	$$
	\{a'\in \Phi_{H,\af}, a'(\fkf)<0\}=\bigcap_{w\in W_{H,\fkf}} \{wa'\in \Phi_{H,\af}, a'(\fka')<0\},
	$$
	 we have an isomorphism 
$$
T_{e}\Gr_{\calG}\cong\left(\bigoplus_{a'=\alpha'+m\in \Phi_{H,\af},a'(\fkf)<0}\bbC X_{\alpha'}u^m\right)^\sigma \oplus(\fkt_H\otimes u^{-1}\bbC[u^{-1}])^\sigma.
$$
As in Proposition \ref{prop: basis of invariants}, we find that $\{e_a\}_{a\in \Phi_{\af},a(\fkf)<0}$ is a basis for the vector space $(\bigoplus_{a'=\alpha'+m\in \Phi_{H,\af},a'(\fkf)<0}\bbC X_{\alpha'}u^m)^\sigma$. Here we use the fact that $\fkf$ is $\sigma$-invariant in order to show that the condition $a(\fkf)<0$ for $a\in \Phi_{\af}$ is equivalent to the condition that $a'(\fkf)<0$, for $a'=\alpha'+m\in \Phi_{H,\af}$ attached to $a'$ as in \S\ref{para: affine root gps}.  The result follows. \end{proof}

\section{Root curves in  twisted affine Grassmannians}

\subsection{} We keep the notations and assumptions of the previous section so that $G$ is an adjoint absolutely simple reductive group, but we will specialize the discussion to the case of twisted affine Grassmannians. Thus let $\fks\in \calA(G,S,F)$ be a fixed special vertex (note that $\fks$ is absolutely special unless $G$ is of type $C-BC_n$). We write $\calG=\calG_{\fks}$ for the parahoric group scheme over $\calO_F$, and let  $\Gr_{\calG}=\Fl_{G,\fkf}$ be the twisted affine Grassmannian associate to $\fks$. 
The choice of special vertex induces an identification $\calA(G,S,F)\cong X_*(T)_I\otimes_{\bbZ} \bbR$ taking $\fks$ to 0.  We fix $\fka\subset \calA(G,S,F)$ the alcove which is contained in the anti-dominant chamber under this identification.

Let $N$ denote the centralizer of $T$ and let $\calT$ denote the connected N\'eron model of $T$ over $\calO_F$. The Iwahori Weyl group is defined to be the quotient 
$$
W:=N(F)/\calT(\calO_F).
$$
 There is a natural short exact sequence
\[
\xymatrix{1\ar[r] &X_*(T)_I\ar[r] & W \ar[r] & W_0\ar[r] & 1},
\]
where $W_0:=N(F)/T(F)$ is the spherical Weyl group. 
For each $\lambda\in X_*(T)_{I}$, we write $t^\lambda$ for the corresponding translation element in $W$. We also write $t^\lambda\in \Gr_{\calG}(\bbC)$ for the image of an element $\tilde{t}^\lambda\in N(F)$ lifting $t^\lambda$; this point is independent of the choice  of $\tilde{t}^\lambda$.

 The affine Weyl group $W_a\subset W$ is defined to be the subgroup generated by the set of simple reflection in the walls of $\fka$. Then we have exact sequences
\[
\xymatrix{1\ar[r] &X_*(T_{\mathrm{sc}})_I\ar[r] & W_a \ar[r] & W_0\ar[r] & 1}
\]
where $T_{\mathrm{sc}}$ is the preimage of $T$ in the simply connected cover of the derived group of $G$.
The choice of alcove $\fka$ determines a length function and Bruhat order on $W$. 
The choice of  special vertex $\fks$ determines splittings of the above exact sequences, and hence semi-direct product decompositions 
$$
W\cong X_*(T)_I\rtimes W_0,\ \ \ \ \ W_a\cong X_*(T_{\mathrm{sc}})_I\rtimes W_0.
$$
There is a reduced root system $\Sigma$ (the \'echelonnage root system) such that 
$$
W_a\cong Q(\Sigma)\times W(\Sigma)
$$ 
where $Q(\Sigma)$  (resp. $W(\Sigma)$) denotes the coroot lattice (resp. Weyl group) for $\Sigma$.

\subsection{}\label{para: echelonnage} By \cite[Theorem B]{HainesDuality}, $\Sigma$ has the following description.  We write $N_I:X^*(T_H)\rightarrow X^*(T_H)$ for the norm map $\alpha\mapsto \sum_{\alpha'\in I\alpha}\alpha'$ and let $\Delta_H$ denote the set of simple roots for $H$ determined by $B_H$.  For $\alpha \in \Delta_H$,
we let $N_I'(\alpha)$ be the modified norm, defined as 
$$
N_I'(\alpha)=\begin{cases} N_I(\alpha) &\text{if all elements of $I\alpha$ are orthogonal,}\\
2N_I(\alpha)&  \text{otherwise.}\end{cases}
$$
Then $\Sigma$ has a basis given by the image of $\Delta_H$ under $N_I'$.

For $\alpha\in \Sigma$ and $k\in \bbZ$, we have the affine function $\alpha+k$ on $\calA(G,S,F)$; here by convention we have $\alpha(\fks)=0$ for $\alpha\in \Sigma$. The  decomposition of $\calA(G,S,F)$ into facets is given by the vanishing hyperplanes of these functions. We write $\Phi_{\af}(\Sigma)=\{\alpha+k|\alpha\in \Sigma,k\in \bbZ\}$; then we obtain an identification $$\psi_{\fks}:\Phi_{\af}(\Sigma)\xrightarrow{\sim}\Phi_{\af}$$ which takes $\alpha+k$ to the affine root  which vanishes on the affine hyperplane $\alpha+k=0$. For $a=\alpha+k\in \Phi_{\af}(\Sigma)$, we have $\calU_{\psi_{\fks}(a)} \subset L^+\calG$ if and only if $k\geq 0$.

 If $\fks=\fks_0$, we have the following relation between the description of $\Phi_{\af}$  given in terms of $\Sigma$ and the description given in terms of the relative roots $\Phi$. Note that for each $\alpha\in \Phi$, there is a unique element $\nu(\alpha)$ of $\Sigma$ proportional to it.
 
 Case (1): Let $\alpha \in \Phi$ with $\frac{1}{2}\alpha,2\alpha\notin \Phi$.  For $a=\alpha+m\in \Phi_{\af}$, we have $\nu(\alpha)=d_\alpha \alpha$ and $\psi_{\fks_0}^{-1}(a)=\nu(\alpha)+d_\alpha m$.
 
 Case (2): Let $\alpha,2\alpha\in \Phi$. For $a=2\alpha+m$, we have $2m$ is an odd integer, and $\psi_{\fks_0}^{-1}(a)=\nu(\alpha)+2m.$ For $a=\alpha+m$, $m\in \frac{1}{2}\bbZ$, we have $\psi_{\fks_0}^{-1}(a)=\nu(\alpha)+4m.$

\subsection{} For $a=\alpha+k\in \Phi_{\af}(\Sigma)$, we  write $\calU_a$ for the affine root subgroup $\calU_{\psi_{\fks}(a)}$.     Then by our assumptions $\calU_a\subset L^+\calG$ if and only if $k\geq 0$. We write $\langle\ ,\ \rangle$ for the pairing between $X_*(T)_{I}\otimes_{\bbZ}\bbR$ and the root lattice associated to $\Sigma$. Then for $\lambda\in X_*(T)_I$ and $a=\alpha+k\in \Phi_{\af}(\Sigma)$, we have 
\begin{equation}\label{eqn: translation elt action}
 t^\lambda\calU_at^{-\lambda}=\calU_{a+\langle \lambda,\alpha\rangle.}
\end{equation}

For each $a\in \Phi_{\af}(\Sigma)$, we define a morphism $f_a:\bbA^1\rightarrow \Gr_{\calG}$ by taking the composition 
$$
\bbA^1\xrightarrow{\tilde{f}_{\psi_{\fks}(a)}}\calU_{\psi_{\fks(a)}}\rightarrow LG\rightarrow \Gr_{\calG}.
$$
We will call (a translate of) $f_a$ a \emph{root curve}.
A simple calculation in $\SL_2$ or $\mathrm{SU}_3$ using the explicit description of the affine root groups $\calU_a$ gives the following proposition (cf. \cite[Prop. 6.2.5]{KZ}).

\begin{prop}\label{prop: SL_2, U_3 calc}
	Let $a=\alpha-k\in \Phi_{\af}(\Sigma)$ with $\alpha\in \Sigma,k \in \bbZ$. Then $f_a(0)=e$, where $e$ is the base point of $\Gr_{\calG}$,  and we have 
	$$
	f_a(x)\in \calU_{-a}t^{-k\alpha^\vee}L^+\calG/L^+\calG,
	$$
	 for $x\in \bbA^1\setminus \{0\}$. Here $\alpha^\vee \in X_*(T_{\sc})_I$ corresponds to the coroot for $\alpha$.\qed
\end{prop}
\subsection{} 
Now let $X_*(T)_I^+$ denote the set of dominant elements in $X_*(T)_I$ with respect to $B$. For $\mu\in X_*(T)_I^+$, we write $\Gr_{\calG,\mu}$ for the open Schubert cell given by the $L^+\calG$-orbit of $t^\mu$ in $\Gr_{\calG}$ , and we write $\Gr_{\calG,\lleq\mu}$ for its closure. Let $\lleq$ denote the dominance order on $X_*(T)_I^+$ with respect to $\Sigma$, i.e. $\lambda\lleq\mu$ if and only if $\mu-\lambda$ is an integral linear combination of positive coroots for $\Sigma$ with non-negative coefficients. By \cite[Corollary 1.8]{Richarz}, this coincides with the Bruhat order under the identification
$W_K\backslash W/W_K\cong X_*(T)_I^+$; here $W_K\subset W$ is the subgroup generated by the simple reflections fixing $\fks$. In particular $\Gr_{\calG,\lleq\lambda}\subset \Gr_{\calG,\lleq\mu}$ if and only if $\lambda\lleq\mu$.

Fix $\mu,\lambda \in X_*(T)_I^+$ with $\mu \ggeq\lambda$.  For an element $\nu\in X_*(T)_I$, we write $\nu_{\dom}\in X_*(T)_I^+$ for the dominant representative of $\nu$. Then for $\alpha\in \Sigma$, we define
$$
k_\alpha=k_\alpha^{(\lambda,\mu)}:=\max\{k\in \bbZ\, |\, (\lambda-k\alpha^\vee)_{\dom}\lleq \mu\}.
$$

\begin{prop}\label{prop: root curve bound relation}Let $\alpha\in \Sigma$.
	\begin{enumerate}\item We have $k_\alpha=k_{-\alpha}+\langle \lambda,\alpha\rangle$.
		\item For $1\leq k\leq k_\alpha$, the root curve $t^\lambda f_{\alpha-k}:\bbA^1\rightarrow \Gr_{\calG}$ has image contained in $\Gr_{\G,\lleq\mu}$ and maps $0\in \bbA^1$ to the point $t^\lambda$.
		\end{enumerate}
\end{prop}
\begin{proof}
(1)	Let $s_\alpha\in W_0$ denote the simple  reflection corresponding to $\alpha$ which acts on $X_*(T)_{I}\otimes_{\bbZ}\bbR$. Then we have \begin{align*}s_\alpha(\lambda-k\alpha^\vee)&=\lambda+k\alpha^\vee-\langle \lambda,\alpha\rangle\alpha^\vee\\ &=\lambda -(k-\langle \lambda,\alpha\rangle)(-\alpha^\vee)\end{align*}
	for any $k\in \bbZ$. 
	
	Since $s_\alpha(\lambda-k\alpha^\vee)_{\dom}=(\lambda-k\alpha^\vee)_\dom$, we have $(\lambda-k\alpha^\vee)_\dom\leq \mu$ if and only if $(\lambda -(k-\langle \lambda,\alpha\rangle)(-\alpha^\vee))_{\dom}\leq \mu$, hence (1) follows.
	
	(2) By Proposition \ref{prop: SL_2, U_3 calc}, we have  $t^\lambda f_{\alpha-k}(0)=t^\lambda$. For the first statement, we first assume $1\leq k\leq \langle \lambda,\alpha\rangle $; note this only occurs when $\alpha$ is a positive root. Then $$t^\lambda f_{\alpha-k}(x)\in \calU_{\alpha- k+\langle\lambda,\alpha\rangle} t^\lambda L^+\calG/L^+\calG\subset \Gr_{\calG, \lleq\lambda}$$ and hence $t^\lambda f_{\alpha-k}$ has image in $\Gr_{\calG, \lleq\mu}$ since $\lambda\lleq\mu$.

Now assume $\langle\lambda,\alpha\rangle \leq k\leq k_\alpha$.	By Proposition \ref{prop: SL_2, U_3 calc}, we have 
$t^\lambda f_{\alpha-k}(x)\in t^\lambda \calU_{-\alpha+k}t^{-k\alpha^\vee}L^+\G/L^+\calG$ for $x\in \bbA^1\setminus\{0\}$,  and hence by \eqref{eqn: translation elt action} $$
t^\lambda f_{\alpha-k}(x)
\in \calU_{-\alpha+k-\langle\lambda,\alpha\rangle}t^{\lambda-k\alpha^\vee}L^+\G/L^+\G.
$$
By assumption $k-\langle\lambda,\alpha\rangle \geq 0$. Thus $\calU_{-\alpha+k-\langle\lambda,\alpha\rangle}\subset L^+\calG$ and  we have 
$$
t^\lambda f_{\alpha-k}(x)\in \Gr_{\calG,  \lleq( \lambda-k\alpha^\vee )_{\dom}}
$$
 for $x\in \bbA^1\setminus\{0\}$.  Now $\lambda-k\alpha^\vee$ lies in the interval $[\lambda-k_\alpha\alpha^\vee,s_\alpha(\lambda-k_\alpha\alpha^\vee)]$,  and hence is contained the convex hull of the $W_0$-orbit of $(\lambda-k_\alpha\alpha^\vee)_{\dom}$. Thus we have 
 $$
 (\lambda-k\alpha^\vee)_{\dom}\lleq(\lambda-k_\alpha\alpha^\vee)_{\dom}\lleq \mu
 $$ 
 and the result follows.
\end{proof}

\subsection{} The root curves above allow us to obtain the following bound on dimensions of tangent spaces to Schubert varieties. Using the isomorphism $\Gr_{\calG}\xrightarrow{\sim}\Gr_{\calG}$ given by multiplication by $t^\lambda$ which takes $e$ to $t^\lambda$, we obtain an identification of tangent spaces $$T_e\Gr_{\calG}\cong T_{t^\lambda}\Gr_{\calG}.$$
Thus by Proposition \ref{prop: tangent space AFV}, we have an identification 
\begin{equation}
\label{eqn: id tangent space t^l}T_{t^\lambda}\Gr_{\calG}\cong \left(\bigoplus_{a\in\Phi_{\af}, a(\fks)<0}\bbC e_a\right)\oplus(\fkt_H\otimes u^{-1}\bbC[u^{-1}])^\sigma.
\end{equation}

For a subspace $V\subset T_{t^\lambda}\Gr_{\calG}$, we write $V^{\mathrm{root}}$ for the intersection of $V$ with the subspace $\bigoplus_{a\in\Phi_{\af}, a(\fks)<0}\bbC e_a\subset T_{t^\lambda}\Gr_{\calG}$.

\begin{cor}\label{cor: tangent dimension bound} Let $\lambda,\mu\in X_*(T)^+_{I}$ with $\lambda\lleq\mu$. 
	$$
	\dim (T_{t^\lambda}\Gr_{\calG,\lleq\mu})^{\mathrm{root}}\geq \sum_{\alpha\in \Sigma} k_\alpha^{(\lambda,\mu)}.
	$$
\end{cor}

\begin{proof}
	Using the bijection $\psi_\fks:\Phi_{\af}(\Sigma)\cong \Phi_{\af}$, we can rewrite \eqref{eqn: id tangent space t^l} as  
	$$
	T_{t^\lambda}\Gr_{\calG}\cong \left(\bigoplus_{\alpha-k\in\Phi_{\af}(\Sigma), k>0}\bbC e_{\psi_{\fks}{(\alpha-k)}}\right)\oplus(\fkt_H\otimes u^{-1}\bbC[u^{-1}])^\sigma.
	$$
	By the construction of the isomorphism in Proposition \ref{prop: tangent space AFV} and the computations in \S\ref{para: affine root gps}, the one-dimensional subspace of $T_{t^\lambda}\Gr_{\calG}$ spanned by the tangent space of the root curve $t^\lambda f_{\alpha-k}$ is $\bbC e_{\psi_{\fks}(\alpha-k)}$. By Proposition \ref{prop: root curve bound relation}, for $1\leq k\leq k_\alpha$, this lies in the subspace $(T_{t^\lambda}\Gr_{\calG,\lleq\mu})^{\mathrm{root}}$. Summing over all $\alpha\in \Sigma$ gives the result.
	
	\end{proof}

\section{Tangent spaces for minimal degenerations}\label{sec: min deg}

\subsection{} \label{subsec: min deg}

The bound obtained in Corollary \ref{cor: tangent dimension bound} can be made more explicit when applied to boundary components of minimal degenerations. 
We keep the notation of the previous section so that $\Gr_{\calG}$ is the twisted affine Grassmannian associated to an adjoint absolutely simple reductive group over $F$ and a choice of special parahoric corresponding to a point $\fks\in \calA(G,S,F)$.

For $\mu,\lambda\in X_*(T)^+_{I}$, we say that $\mu\rightsquigarrow \lambda$ is a minimal degeneration if
\begin{enumerate}
	\item 
	$\mu \ggeq\lambda$ and $\mu \neq \lambda$.
	\item If $\nu\in X_*(T)_I^+$ is such that $\mu\ggeq\nu\ggeq\lambda$, then $\mu=\nu$ or $\lambda=\nu$.
\end{enumerate}

For $\lambda\lleq\mu$, we may write $\mu-\lambda$ uniquely as a linear combination of  simple coroots $\alpha^\vee\in \Sigma^\vee\subset X_*(T_{\mathrm{sc}})_I$ with non-negative integral coefficients. We write $\supp(\mu-\lambda)$ for the sub-root system of $\Sigma$ generated  by the roots associated to the set of simple coroots which appear with positive coefficients.

We have the following classification of minimal degenerations due to Stembridge \cite[Theorem 2.8]{Stembridge}.

\begin{prop}\label{prop: Stembridge}
	Let $\mu\rightsquigarrow \lambda$ be a minimal degeneration. Then one of following holds:
	\begin{enumerate} \item $\mu-\lambda$ is a simple coroot.

		\item $\mu-\lambda$ is the short dominant coroot of $\mathrm{supp}(\mu-\lambda)$, and $\lambda=0$ on $\mu-\lambda$.
		
		\item $\mu-\lambda$ is the short dominant coroot of $\mathrm{supp}(\mu-\lambda)$; $\mathrm{supp}(\mu-\lambda)$ is of type $C_n$ and $\lambda$
		 on $\mathrm{supp}(\mu-\lambda)$ is 
		$$
		0\,\rule[0.1cm]{0.4cm}{0.5pt}\,0\,\rule[0.1cm]{0.4cm}{0.5pt}\,\cdots\,\rule[0.1cm]{0.4cm}{0.5pt}\,0\xLeftarrow{\ \ \ }  1\ .
		$$
		
		\item $\supp(\mu-\lambda)$ is of type $G_2$, and $\lambda=(2\xRrightarrow{\ \ \ } 0)$, $\mu=(1\xRrightarrow{\ \ \ } 1)$.
		
				\item $\supp(\mu-\lambda)$ is of type $G_2$, and $\lambda=(1\xRrightarrow{\ \ \ } 0)$, $\mu=(0\xRrightarrow{\ \ \ } 1)$.
		\end{enumerate}
	Here the integer at a vertex of the Dynkin diagram represents the pairing with the simple root corresponding to that vertex.
\end{prop}

\subsection{} Let $\Sigma^+$ denote the subset of positive roots, and we let $\rho$ denote the half sum of positive roots in $\Sigma$. 

\begin{prop}\label{prop: root curves min deg}
	Let $\mu\rightsquigarrow \lambda$ be a minimal degeneration. Then 
	$$
	\dim (T_{t^\lambda}\Gr_{\calG,\lleq\mu})^{\mathrm{root}}\geq \dim\Gr_{\calG,\lleq\mu}.
	$$
\end{prop}

\begin{proof}
By \cite[Corollary 2.10]{Richarz}, we have $\dim\Gr_{\calG,\lleq\mu}=\langle\mu,2\rho\rangle$. Thus by Corollary \ref{cor: tangent dimension bound},  it suffices to show 
\begin{equation}
\label{eqn: combinatorial dimension} \sum_{\alpha\in \Sigma}k_\alpha\geq \langle\mu,2\rho\rangle. 
\end{equation}
	By Proposition \ref{prop: root curve bound relation}, we have $k_{\alpha}=k_{-\alpha}+\langle\lambda,\alpha\rangle$ for any $\alpha\in \Sigma$. Substituting this in for each $\alpha \in \Sigma^+$, equation  \eqref{eqn: combinatorial dimension} becomes
	 \begin{equation}\sum_{\alpha\in \Sigma^+}2k_{-\alpha}+\langle\lambda,2\rho\rangle\geq \langle\mu,2\rho\rangle.
	\end{equation}

	We consider the fives cases in Proposition \ref{prop: Stembridge} separately. Note that in each case $\mu-\lambda$ is a positive  coroot in $\Sigma^\vee$ which we call $\beta^\vee$. Hence it suffices to show that 
	$$
	\sum_{\alpha\in \Sigma^+}2k_{-\alpha}\geq \langle \beta^\vee,2\rho\rangle.
	$$
	\begin{enumerate}
		\item  In this case, $\beta$ is a simple root. Then $\langle \beta^\vee,2\rho\rangle=2$ since $\rho$ is equal to the sum of the fundamental weights for $\Sigma$ (cf. \cite[Chap. VI]{Bourbaki}).  Thus it suffices to find  some $\alpha\in \Sigma^+$ with $k_{-\alpha}\geq 1$. It is easy to see that $k_{-\beta}\geq1$ works.

		\item Let $\Sigma_J$ denote the sub-root system  $\supp(\mu-\lambda)$, and $W_J\subset W(\Sigma)$ its Weyl group. If $\gamma^\vee$ is a short positive coroot in $
		\Sigma_J^\vee$, then there is a $w\in W_J$ such that $w(\gamma^\vee)=\beta^\vee$. Then $w(\lambda+\gamma^\vee)=\lambda+\beta^\vee=\mu$ and hence $(\lambda+\gamma^\vee)_{\dom}\lleq\mu$, i.e. $k_{-\gamma}\geq1$. We write $\Sigma_{J,\mathrm{short}}^{\vee,+}$ for the set of short positive coroots in $\Sigma_J^\vee$. Then we have 
		$$
		(\sum_{\alpha\in \Sigma^+} 2k_{-\alpha})\geq 2|\Sigma_{J,\mathrm{short}}^{\vee,+}|
		$$
		and hence it suffices to show that $$2|\Sigma_{J,\mathrm{short}}^{\vee,+}|\geq \langle\beta^\vee,2\rho\rangle.$$
		
		Recall that $\beta^\vee$ is the short dominant coroot in $\Sigma_J$. We write $\beta^\vee$ as a sum of simple coroots which necessarily lie in $\supp(\mu-\lambda)$, and we let  $N_\beta$ be the number of simple coroots appearing (counted with multiplicity). Then $\langle \beta^\vee, 2\rho\rangle=2N_\beta$. We thus reduce to showing $$|\Sigma_{J,\mathrm{short}}^{\vee,+}|\geq N_\beta.$$
Note that this only depends on the root system $\Sigma_J$. 

To show this last inequality, let $\gamma^\vee$ be a short simple coroot and $v$ a minimal length element in $W_J$ such that $v(\gamma^\vee)= \beta^\vee$. We let $s_ns_{n-1}\dotsc s_1$ be a reduced word decomposition for $v$ and write $v_i=s_i\dotsc s_1$ for $i=1,\dotsc,n$. Then $n+1=N_\beta$, and $\gamma^\vee, v_1(\gamma^\vee),\dotsc,v_n(\gamma^\vee)$  give $n+1$ distinct elements of $\Sigma_{J,\mathrm{short}}^{\vee,+}$.

\item Let $\Sigma_J$ denote the sub-root system $\supp(\mu-\lambda)$ which is of type $C_n$. Let $\gamma^\vee$ be a positive short root. Then there is a $w\in W_J$ such that $w(\lambda)=\lambda$ and $w(\gamma^\vee)=\beta^\vee$. We can prove this explicitly as follows.

We identify the coweight lattice of $\Sigma_J$ with  a submodule of $\frac{1}{2}\bbZ^n$. Let $e_1,\dotsc,e_n$ be the standard basis of $\bbZ^n$. Then the coroots are given by$$\pm e_i,i=1,\dotsc,n
\text{ and } \pm e_i\pm e_j, 1\leq i<j\leq n.$$ We choose the ordering so that the positive coroots are given by $e_i,i=1,\dotsc,n$ and $e_i\pm e_j, 1\leq i < j\leq n$. Thus the short positive coroots are given by $e_1,\dotsc,e_n$ with $\beta^\vee= e_1$, and each can be made dominant using $w\in S_n\subset W_J$. By the condition on $\lambda$, the subgroup $S_n$ fixes $\lambda$. Thus $w(\lambda+\gamma^\vee)=\lambda+\beta^\vee=\mu$ and hence $(\lambda+\gamma^\vee)_{\dom}\lleq \mu$, i.e. $k_{-\gamma }\geq1$. The same argument as case (2) proves the desired inequality.

\item[(4)+(5)]Let $\supp(\mu-\lambda)=\{\alpha_1,\alpha_2\}$ with $\alpha_1$ the short root. Then in both cases, we have $$\mu-\lambda=\beta^\vee=\alpha_1^\vee+\alpha_2^\vee.$$ Thus $\langle\beta^\vee,2\rho\rangle=4$ and it suffices to show that there exist two positive coroots $\gamma^\vee_1$ and $\gamma^\vee_2$  such that $(\lambda+\gamma_1^\vee)_{\dom}=(\lambda+\gamma_2^\vee)_{\dom}=\mu$, in which case we have $k_{-\gamma_1}\geq1$ and $k_{-\gamma_2}\geq1$. Clearly we can take $\gamma^\vee_1=\beta^\vee$. Note that $s_{\alpha_1}(\lambda)=\lambda$, and $s_{\alpha_1}(\alpha_2^\vee)=\alpha_1^\vee+\alpha_2^\vee$. Thus $(\lambda+\alpha_2^\vee)_{\dom}=\mu$ and we can take $\gamma_2^\vee=\alpha_2^\vee$.
\end{enumerate}\end{proof}

\subsection{} We now prove our main theorem which verifies a conjecture of Haines--Richarz \cite[Conjecture 5.4]{HainesRicharz}.

\begin{thm}\label{thm: Haines-Richarz}
Let $\Gr_{\calG}$ be the twisted affine Grassmannian associated to an absolutely special vertex $\fks_0$. Let $\mu\in X_*(T)^+_{I}$. Then the smooth locus of $\Gr_{\calG,\lleq\mu}$ is $\Gr_{\calG,\mu}$. 
\end{thm}

\begin{proof} Clearly $\Gr_{\calG,\mu}\subset \Gr_{\calG,\lleq\mu}$ is smooth. For the converse, note that the smooth locus is open, hence if $\lambda\lleq\mu$ is such that $\Gr_{\calG,\lambda}\subset \Gr_{\calG,\lleq\mu}^{\mathrm{sm}}$, then $\Gr_{\calG,\nu}\subset \Gr_{\calG,\lleq\mu}^{\mathrm{sm}}$ for $\lambda\lleq\nu\lleq\mu$. Thus it suffices to show for $\mu\rightsquigarrow\lambda$ a minimal degeneration that $\Gr_{\calG,\lambda}$ does not lie in the smooth locus of $\Gr_{\calG,\lleq\mu}$. Using the $L^+\calG$-action, it suffices to show $\Gr_{\calG,\lleq\mu}$ is singular at $t^\lambda$, i.e. that 
$$
\dim T_{t^\lambda}\Gr_{\calG,\lleq\mu}>\dim\Gr_{\lleq\mu}=\langle \mu, 2\rho\rangle.
$$ 
By Proposition \ref{prop: root curves min deg}, we have  
	$$
	\dim (T_{t^\lambda}\Gr_{\calG,\lleq\mu})^{\mathrm{root}}\geq \dim \Gr_{\calG,\lleq\mu}.
	$$
Thus it suffices to show  there exists $X\in T_{t^\lambda}\Gr_{\calG,\lleq\mu}$ with non-zero component in $(u^{-1}\bbC[u^{-1}]\fkt_H)^\sigma$.

Note that the restriction of the  $L^+\calG$ action on $\Gr_{\calG,\lleq\mu}$ to the subgroup $$\calG^\lambda:=\calG(\bbC[[t]])\cap t^\lambda \calG(\bbC[[t]])t^{-\lambda}$$  preserves the point $t^\lambda$. Thus we obtain an action $\star$ of $\calG^\lambda$ on $T_{t^\lambda}\Gr_{\calG,\lleq\mu}$. Under the identification \eqref{eqn: id tangent space t^l}, the action on the right is given by composing the natural adjoint action with the map $g\mapsto t^{-\lambda} g t^{\lambda}$.
Now since $\langle \mu,2\rho\rangle>\langle \lambda,2\rho\rangle$, we have $k_{-\alpha}\geq1$ for some $\alpha\in \Sigma^+$ and hence $e_{a}\in T_{t^\lambda}\Gr_{\calG,\lleq\mu}$ for $a=\psi_{\fks_0}(-\alpha-1)$.We set $h':=\tilde{f}_{\psi_{\fks_0}(\alpha+\langle\lambda,\alpha\rangle)}(1)\in \calG^\lambda$ and $h:=t^{-\lambda} h't^\lambda=\tilde{f}_{\psi_{\fks_0}(\alpha)}(1)$. We claim that
$$
X:=h'\star e_{\psi_{\fks_0}(-\alpha-1)}=\mathrm{Ad}(h)e_{\psi_{\fks_0}(-\alpha-1)}\in T_{t^\lambda}\Gr_{\calG,\lleq\mu}
$$
has non-zero component along $(u^{-1}\bbC[u^{-1}]\fkt_H)^\sigma$. 
 
 Note that since $\fks_0$ is absolutely special, we have $\psi_{\fks_0}(\alpha)=\alpha'\in \Phi_{\af}$ with $\alpha'\in \Phi$ a relative root and $\frac{1}{2}\alpha'\notin \Phi$ (cf. \S\ref{para: echelonnage}). Then 
 $$
 \psi_{\fks_0}(-\alpha-1)=\begin{cases}\alpha'-\frac{1}{d_{\alpha'}} &\text{if $2\alpha'\notin\Phi$,}\\
 2\alpha'-\frac{1}{2} &\text{otherwise.}\end{cases}$$
 
Then by Proposition \ref{prop: Cartan direction}, $X$ has a non-zero component in $(u^{-1}\bbC[u^{-1}]\fkt_H)^\sigma$ as desired.
\end{proof}
\begin{remark}\label{rem: general group}
A standard argument can be used to extend the theorem to a general reductive group $G$ over $\bbC\llps t\lrps$.	We write $G_{\ad}=\prod_{i=1}^r \mathrm{Res}_{F_i/F}G_i$, where $G_i$ is an adjoint absolutely simple group over a finite extension $F_i$ of $F$. In this case  a vertex $\fks_0$ in the building $\calB(G,F)$ for $G$ over $F$ is said to be absolutely special if it maps to an absolutely special vertex in $\calB(\mathrm{Res}_{F_i/F}G_i, F)\cong\calB(G_i,F_i)$ for all $i$. Then if $\calG/\bbC\lps t\rps$ corresponds to an absolutely special vertex $\fks_0$, any Schubert variety of $\Gr_{\calG}$ is isomorphic to a product of Schubert varieties in $\Gr_{\calG_i}$, where $\calG_i$ corresponds to the image of $\fks_0$ in $\calB(G_i,F_i)$. It follows that the smooth locus is the open stratum.
	
\end{remark}

\subsection{} Here we explain an alternative proof of Theorem \ref{thm: Haines-Richarz} that uses some extra algebro-geometric inputs, but which avoids the combinatorics in  Proposition \ref{prop: root curves min deg}.

Let $\mu\rightsquigarrow \lambda$ be a minimal degeneration and set $Y:=\Gr_{\calG,\mu}\cup \Gr_{\calG,\lambda}\subset \Gr_{\calG}$. Let $T':=T_H^\sigma$ be the fixed point scheme of  $T_H$ over $\bbC$ under the action of $\sigma$. Then $T'\subset L^+\calG$ and hence acts on $Y$. By  \cite[Theorem 4.2, Corollary 4.3]{BH}, the fixed point scheme $Y^{T'}$ 
has reduced locus $(Y^{T'})_{\rm red}$ which is $0$-dimensional and consists of just the points $t^\mu$ and $t^\lambda$.
 
 Suppose now that $Y$ is smooth.  Then $Y^{T'}$ is smooth by \cite{Iversen}, and hence reduced and equal to $\{t^\mu, t^\lambda\}$, so $t^\lambda$ is an isolated $T'$-fixed point in $Y$. Then, the tangent space
 $T_{t^\lambda}Y=T_{t^\lambda}\Gr_{\calG,\lleq\mu}$ is a $T'$-module whose weight zero direct summand $(T_{t^\lambda}Y)_0$ is trivial. This implies that $T_{t^\lambda}\Gr_{\calG,\lleq\mu}=(T_{t^\lambda}\Gr_{\calG,\lleq\mu})^{\rm root}$.
Now $\mu-\lambda=\beta^\vee$, for some $\beta\in \Sigma^+$,  and hence $k_{-\beta}^{(\lambda,\mu)}\geq 1$. Then as in the proof of Theorem \ref{thm: Haines-Richarz},  taking  $h'=\tilde{f}_{\psi_{\fks_0}(\beta+\langle\lambda,\beta\rangle)}(1)$,   we obtain an element 
$$
X=h'\star e_{\psi_{\fks_0}(-\beta-1)}\in T_{t^\lambda}\Gr_{\calG,\lleq\mu} 
$$
whose component in $(u^{-1}\bbC[u^{-1}]\fkt_H)^\sigma$ is non-trivial. Then $X\not\in (T_{t^\lambda}\Gr_{\calG,\lleq\mu})^{\rm root}$ which gives a contradiction.

 \subsection{Example: The ramified triality}
 We explain how our computations can be used to give a lower bound on the dimension of the tangent space at the worst stratum of the quasi-minuscule Schubert variety for the ramified triality. This recovers a result of Haines--Richarz  \cite[Theorem 8.1]{HainesRicharz}.
 
 In our notation we can take $H=\mathrm{Spin}_8$.  The absolute root system is of type $D_4$ with simple roots $\alpha_i, i=1,2,3,4$; here we use the notation from Bourbaki \cite{Bourbaki}. Then $\sigma_0$ fixes $\alpha_2$ and acts as the cycle $1\rightarrow 3\rightarrow 4\rightarrow 1$ on the other roots. The $\sigma_0$-orbits on positive roots can be listed as follows: 
 $$
 \{\alpha_2\}, \{\alpha_1,\alpha_3,\alpha_4\} , \{\alpha_1+\alpha_2,\alpha_2+\alpha_3,\alpha_2+\alpha_4\},
 $$
 $$
 \{\alpha_1+\alpha_2+\alpha_3,\alpha_2+\alpha_3+\alpha_4,\alpha_1+\alpha_2+\alpha_4\},
 $$
 $$
 \{\alpha_1+\alpha_2+\alpha_3+\alpha_4\}, \{\alpha_1+2\alpha_2+\alpha_3+\alpha_4\}.
 $$

 The simple roots for the \'echelonnage root system $\Sigma$ are the modified norms of the simple absolute roots, so a basis for $\Sigma$ is the pair $\alpha=\alpha_2,\beta:=\alpha_1+\alpha_3+\alpha_4$ which is of type $G_2$. The highest root is $\tilde{\alpha}=3\alpha+2\beta$.
 
 The coroots are given by 
 $$
 \alpha^\vee=\alpha_2^\vee,\quad \beta^\vee=\frac{\alpha_1^\vee+\alpha_3^\vee+\alpha_4^\vee}{3}
 $$
 and we have $\tilde{\alpha}^\vee=\alpha^\vee+\beta^\vee$. Now take $\mu=\tilde{\alpha}^\vee$ and $\lambda=0$.
 We obtain 6 tangent directions of $\Gr_{\calG,\lleq\mu}$ at the base point $\Gr_{\calG,\lambda}=\{e\}$ by using the root curves
 $$
 x\mapsto f_{\pm\beta-1}(x):=U_{\alpha_1}(u^{-1}x)U_{\alpha_3}(u^{-1}\zeta x)U_{\alpha_4}(u^{-1}\zeta^2 x),
 $$
 $$x\mapsto f_{\pm(\beta+3\alpha)-1}(x):=U_{\alpha_1+\alpha_2}(u^{-1}x)U_{\alpha_2+\alpha_3}(u^{-1}\zeta x)U_{\alpha_2+\alpha_4}(u^{-1}\zeta^2 x),$$
  $$x\mapsto f_{\pm(2\beta+3\alpha)-1}(x):=U_{\alpha_1+\alpha_2+\alpha_3}(u^{-1}x)U_{\alpha_2+\alpha_3+\alpha_4}(u^{-1}\zeta x)U_{\alpha_1+\alpha_2+\alpha_4}(u^{-1}\zeta^2 x),$$ where $\zeta$ is a third root of unity.
  
  Let $e_1,e_2,e_3,e_4$ be a basis for the Cartan $\fkt_H$  so that $e_i-e_{i+1}$ corresponds to the root $\alpha_i$. Then conjugating the root curve $x\mapsto f_{-\beta-1}(x)$ by $ h:=\tilde{f}_{\beta}(1)$ gives an extra tangent direction along the Cartan given by 
  $$
  u^{-1}((e_1 -e_2)+\zeta(e_3-e_4)+\zeta^2(e_3+e_4))\in (u^{-1}\fkt_H)^\sigma.
  $$In fact $(u^{-1}\fkt_H)^\sigma$ is 1-dimensional and the tangent direction along the Cartan given by conjugating the other two root curves above will be the same. Thus we get a lower bound of 7 for the dimension of the tangent space. On the other hand $\langle\mu,2\rho\rangle=6$, hence $\dim\Gr_{\calG,\lleq\mu}=6$. So, $\Gr_{\calG,\lleq\mu}$ is singular at the point $\Gr_{\calG,\lambda}=\{e\}$.

\bibliography{bibfile}

\bibliographystyle{amsalpha}

\end{document}